\documentclass[a4paper,11pt,twoside]{article}
\usepackage{amsmath,amssymb,amsthm,eucal,graphics}
\usepackage[cp1251]{inputenc}
\usepackage[english]{babel}
\usepackage{amsfonts,amssymb,eucal}
\usepackage{amsbsy}
\usepackage{latexsym}

\makeatletter
\gdef\firstpage{1}

\def\year{}
\def\frsthdr{}

\def\firstpageone{0\thepage}
\def\firstpagetwo{00\thepage}
\def\firstpagethree{000\thepage}
\def\firstpagemark{\ifnum\firstpage <10  \firstpageone
\else\ifnum\firstpage<100 \firstpagetwo \else \ifnum\firstpage
<1000 \firstpagethree \else \firstpageone\fi\fi\fi}

\def\footline{\ifnum\thepage=\firstpage \footlineone
\else\footlinetwo\fi}
\def\footlineone{\noindent \footnotesize \sf  \ \year \hspace{\fill}  \hbox{}}
\def\footlinetwo{ }

\def\titles{{Cohomology of the conformal algebra $Cend_{1,x}$}}
\def\authors{{Roman A.~Kozlov}}

\def\oddhedr{\ifnum\thepage=\firstpage \firsthdr \else \odhdr \fi}
\def\firsthdr{\hspace{\fill} \sl \frsthdr \hspace{\fill}\hbox{}}
\def\odhdr{\hspace{\fill}\sl\rightmark \titles \hspace{\fill}
\rm \thepage}

\def\evnhedr{\ifnum\thepage=\firstpage \firsthdr \else \evhdr \fi}
\def\evhdr{\noindent \rm \thepage\hspace*{\fill} \sl\leftmark
\authors \hspace*{\fill}\hbox{}}

\def\ps@newpstyle{\def\@oddhead{
\hspace{-0.65em} \vbox{\oddhedr\vskip 1mm \hrule width
\textwidth}
}
\def\@evenhead{
\hspace{-0.65em} \vbox{\evnhedr\vskip 1mm \hrule width
\textwidth}
}\textsc{}
\def\@oddfoot{\footline}
\def\@evenfoot{\@oddfoot}
}

\def\refer
{
\begin{small}
\begin{enumerate}
\leftmargin=0pt
\rightmargin=0pt
\itemsep=0pt
\parsep=0pt
}

\def\endref{\end{enumerate}\end{small} }
\newtheorem{lemma}{Lemma}
\newtheorem{theorem}{Theorem}
\def\Ker{\mathop{\fam 0 Ker}\nolimits}
\def\End{\mathop{\fam 0 End}\nolimits}

\begin{document}

\setcounter{page}{\firstpage}
\pagestyle{newpstyle}

\sloppy
\rm

\begin{flushright}
UDC~$512.55$
\end{flushright}

\begin{center}\textbf{Hochshild cohomology of the associative conformal
algebra $Cend_{1,x}$}\end{center}

\begin{center}Roman A.~KOZLOV\end{center}
Sobolev Institute of Mathematics, Novosibirsk State University, \\
630090, Novosibirsk, Russia \\
E-mail address: \textbf{KozlovRA.NSU@yandex.ru}
\\
\\
It is established in this work that second Hochshild cohomology group of the
associative conformal algebra $Cend_{1,x}$ is zero. As a corollary, this algebra
split off in each extension with a nilpotent kernel.

\section{Introduction and preliminary information}

\subsection{Introduction}

The axiomatic quantum field theory was originated in 50th of the last century in A. Whiteman's articles etc. Fields with a conformal simmetry were considered in [5]. It can be represented as infinite in both side formal power series whose coefficients are linear operators. A key role in the conformal field theory is played by construction of an operator product expansion (OPE) an  algebraical description of which lead us to a vertex operator algebra (VOA) concept. An axiomatic definition of vertex algebra was introduced by R. Borcherds in [6].

V.G. Kac established [see 9] that coefficients of a singular part of the OPE allow us to describe a commutator of formal power series responding for VOA multiplication operators. In case of an algebraical description of these structures there appears a concept of conformal Lie algebras.  For these algebras the structure theory was constructed in [8] and the representations theory respectively in [7]. During the process of studying representations of the finite type there appear natural algebraic structures playing the same role with respect to conformal Lie algebras as common associative algebras to Lie algebras. These structures are named associative conformal algebras [see 7].

Let $ H = \mathbb F[D]$ be a common polynomial algebra over a field  $\mathbb F$ with a zero char, $\mathbb \mathbb Z_+$ are non-negative integers, $D^{(n)}$ is for $\frac{1}{n!} D^n$, $n\in \mathbb Z_+$ and $D^{(n)} = 0$,  $n<0$.

DEFINITION. {\textit{A conformal algebra}} is a left $H$-module $A$ equipped with a countable number of  $\mathbb F$-linear maps
$\circ_n: A\otimes A \rightarrow A$, $n\in \mathbb Z_+ $ satisfying axioms:

\begin{enumerate}
\item$a\circ_n b = 0$ for sufficiently large $n$,
\item$Da\circ_n b = - na\circ_{n-1}b$,
\item$a\circ_nDb = D(a\circ_nb) + na\circ_{n-1}b$,
\end{enumerate}
where $a,b\in A, n\in \mathbb Z_+$.

First axiom allows us to determine a locality function $N$: $A\times A\rightarrow \mathbb Z_+$ by the rule: $N( a, b ) = \min\{n\in \mathbb Z_+ | a\circ_k b = 0 ~\forall k\ge n\}, ~a, b\in A.$

A conformal algebra is called {\textit{associative}} if for every $a, b, c\in A$ holds 
\begin{center} 
$(a\circ_nb)\circ_mc = \sum\limits_{s\ge0}(-1)^s\binom{n}{s}a\circ_{n-s}(b\circ_{m+s}c),~ m,n\in \mathbb Z_+$,
\end{center} 
or equivalent condition
\begin{center} 
$a\circ_n(b\circ_mc) = \sum\limits_{t\ge0}\binom{n}{t}(a\circ_{n-t}b)\circ_{m+t}c,~ m,n\in \mathbb Z_+$.
\end{center}

Construction of simple and semisimple associative algebras of the finite type was described in [8] (see also [12]). It is established in [13] that analogue of the Wedderburn theorem about a splitting radical (in reference to Lie conformal algebras the same statement, which is analogue of the Levi - Maltsev theorem for Lie algebras, is incorrect [4])  is fulfilled for associative conformal algebras of the finite type.

Associative conformal algebras with a faithful representation of the finite type form more wide class in comparison with algebras of the finite type. All such algebras is subalgebras of $Cend_n$ (the algebra of conformal linear maps of free $n$-generated  $\mathbb F[D]$-module)
\[
Cend_n = M_n(\mathbb F[D, x]),  
\]
where $\circ_m: Cend_n\otimes Cend_n \rightarrow Cend_n$ is determined by a common matrix differentiation, i.e. 
$A(x)\circ_m B(x) = A(x)B^{(m)}(x)$.

Simple and semisimple subalgebras of $Cend_n$ were desribed in [10] and there also was proved that every such algebra $C$
has a nilpotent ideal (radical) defining a semisimple quotient with a faithful representation of the finite type.

The radical splitting problem of this class of algebras was considered in [11, see also 1, 2]. There is a definition of Hochschild cohomology for associative conformal algebras in [1, 2] and in a context of pseudoalgebras [3] (slightly different from the definition in [4]). Cohomology play a key role in the radical splitting problem. Specifically, to proof analogue of the Wedderburn radical splitting theorem for a conformal algebra $A = C/R$ it's enough to show that for arbitrary conformal $A$-bimodule $M$ second Hochschild cohomology group $H^2(A,M)$ is zero.

It is shown in [1] that for $A=Cend_{1}$ and for every bimodule $M$ next condition $H^2(A,M)=0$ is fulfilled. Contrariwise, it follows from the example [see 11] that there exist $A$-bimodule $M$ for $A=Cend_{1,x^2} = x^2Cend_1$ such that $H^2(A,M)\ne 0$.

In this paper we consider an intermidiate case: $A=Cend_{1,x} = xCend_1$. 
This associative conformal algebra is the universal enveloping algebra for the simple conformal Virasoro algebra with locality 2. The main result of this paper is next

\textbf{THEOREM.} \textit{ Let $C = Cend_{1,x}$ and let $M$ be arbitrary conformal bimodule over $C$. Then $H^2(C,M)=0$.}

\subsection{Modules over conformal algebras}

DEFINITION. A {\textit{left module}} over an associative conformal algebra $A$ is an $H$-module $V$, closed with respet to the left action by elements of $A$. Specifically, there is determined such linear maps $\circ_n: A\otimes V\rightarrow V,~ n\in \mathbb Z_+$ that for every $a, b\in A$, $v\in V$, $m\in \mathbb Z_+$

(1) $a\circ_n v = 0$ for sufficiently large $n$,

(2) $Da\circ_n v = - na\circ_{n-1}v$,

(3) $a\circ_nDv = D(a\circ_nv) + na\circ_{n-1}v$, 

(4) $(a\circ_mb)\circ_nv = \sum\limits_{s\ge0}(-1)^s\binom{m}{s}a\circ_{m-s}(b\circ_{n+s}v)$.

Hereinafter a left module over an associative conformal algebra $A$ is a simply left $A$-module. A right $A$-module could be defined in the same way.

A {\textit{bimodule}} is a simultaneously left and right $A$-module with one additional axiom:

\begin{center}
$(a\circ_mv)\circ_nb = \sum\limits_{s\ge0}(-1)^s\binom{m}{s}a\circ_{m-s}(v\circ_{n+s}b).$
\end{center}

\section{Main definitions and formulation of the problem}
\subsection{Connection with pseudoalgebras}

An algebra $H = \mathbb F[D]$ is Hopf's algebra with a comultiplication $\Delta: H \rightarrow H\otimes H, ~\Delta(D) = D\otimes1 + 1\otimes D$, a counit $\varepsilon: H \rightarrow \mathbb F,~\varepsilon(D) = 0$ and an antipode $S: H \rightarrow H,~S(D) = - D$ where $\Delta$, $\varepsilon$ and $S$ are homomorhpisms of algebras. Assuming $\Delta_1 = \Delta$ define by induction 
$\Delta_n = (\Delta\otimes id \otimes \dots  \otimes id)\Delta_{n-1}$,~ $n>1$.

Because of conformal algebras are objects of the pseudotensor category of left $H$-modules, it could be constructed by the pseudotensor language.

One can define the right action $H$ on $H^{\otimes n}$ by the rule:
\[
(h_1\otimes h_2 \otimes \dots \otimes h_n) h= (h_1\otimes h_2 \otimes \dots \otimes h_n)\Delta_{n-1}(h).
\]
Then $H^{\otimes n}$ is a right $H$-module.

Let $V$ be a left unital $H$-module endowed with a bilinear operation $*: A\otimes A \rightarrow (H\otimes H)\otimes_H A$. Such $H$-module $A$ is called an $H$-\textit{pseudoalgebra} and the operation $*$ is a \textit{pseudomultiplication}.

Every conformal algebra $A$ is a $H$-pseudoalgebra and next equality 
\[
 a* b 
=\sum\limits_{n\ge0}(-D)^{(n)}\otimes1\otimes_H(a\circ_nb), \quad a,b\in A,
\]
set up a correspondence between the pseudomultiplication $*$ and operations $\circ_n$, $n\in \mathbb Z_+$.

The definition of a module (bimodule) over an associative conformal algebra along with
the associativity property could be represented through pseudomultiplication-like operations, 
see datails in [1,3]. For research purposes we translate main definitions
in [1](a cochain, cocycle and coboundary) from the pseudoalgebras language to the 
$\circ_n$-multiplications language,~$n\in Z_+$.

\subsection{Definition of the Hochschild cohomology group}

DEFINITION. An {\textit{n-cochain}} with coefficients in $V$ is a $H$-polylinear map $\varphi: A^{\otimes n}\rightarrow (H^{\otimes 
n})\otimes_H V$, i.e.

\begin{center}
$\varphi (h_1a_1\otimes\dots \otimes h_na_n) = (h_1\otimes\dots \otimes h_n\otimes_H 
1) \varphi ( a_1\otimes\dots \otimes a_n) .$
\end{center}

Denote by $C^n (A,V)$ a set of every $n$-cochains of the algebra $A$ and with coefficients in the bimodule $V$. In this paper only 1- and 2-cochains are interesting for us. Obviously, an 1-cochain $\tau $ is arbitrary $H$-polylinear map $\tau : A\to V$.
For 2-cochains  we use the natural representation:
\begin{equation}\label{eq: natrepr}
\varphi(a, b) = \sum\limits_{s\ge0}((-D)^{(s)} \otimes1)\otimes_H \varphi_s(a, 
b), \ \text{where}\ \varphi_s(a, b)\in V.
\end{equation}
Thus a 2-cochain $\varphi $ is a collection of bilinear maps $\varphi_s: A\otimes A\to V$, $s\in \mathbb Z_+$, such that 

\begin{enumerate}
\item$\varphi_s(Da,  b) = - s\varphi_{s-1}(a, b)$,
\item$\varphi_s(a, Db) = D\varphi_s(a, b) + s\varphi_{s-1}(a, b)$,
\item$\varphi_s(a, b) = 0$ for sufficiently large $s$.
\end{enumerate}

Similarly, a 3-cochain $\psi $ is a 2-parametric collection of linear maps $\psi_{mn}: A\otimes A\otimes A \to V$, $n,m\in \mathbb Z_+$ with the same properties.

DEFINITION. A {\textit{differential}} is a map $\delta_n: C^n (A,V)\rightarrow  C^{n+1} (A,V)$  given by the rule $(\delta_n \varphi)( a_1,\dots , a_{n+1}) = a_1 * \varphi(a_2, \dots , a_{n+1}) + \sum\limits_{i = 1}^{n} 
(-1)^i\varphi(a_1,\dots , a_i * a_{i+1},\dots , a_{n+1}) +  (-1)^{n+1}\varphi(a_1,\dots , a_n) * a_{n+1}$.

Translating this definition to the language of $\circ_n$-operations, $n\in \mathbb Z_+$, we obtain the differential for the 1-cohain $\tau $
\[
 (\delta_1\tau)_s    (a,b) = \tau(a)\circ_{n} b -\tau(a\circ_{n} b) + a\circ_n \tau(b),
\]
and for the 2-cohain $\varphi$
\[
 (\delta_2\varphi)_{mn}(a,b,c) = a\circ_m\varphi_n(b,c) + \varphi_m(a, b\circ_n c) -
 \]\[
\sum\limits_{s\ge0}^{} (^m_s)(\varphi_{n+s}(a\circ_{m-s}b, c) + \varphi_{m-s}(a, b)\circ_{n+s}c),~a,b,c\in M,~n, m\in Z_+  
\]

Say an n-cochain $\varphi$ is called an {\textit{n-cocycle}} if $\delta_n\varphi = 0$. A cocycle $\varphi\in Z^n(A,V)$ is called an {\textit{n-coboundary}} if there exists an (n-1)-cochain $\psi $ such that $\varphi = \delta_{n-1}\psi$.
Denote by $Z^n ( A,V)$ a set of every n-cocycles and by $B^n (A,V)$ a set of every n-coboundaries.

DEFINITION. The {\textit{Hochschild cohomology group}} of an associative conformal algebra $A$ on an $A$-bimodule $V$ is $H^n (A,V) = Z^n ( A,V) /B^n (A,V)$.

\subsection{Extensions. Connection with second cohomology group}
 
DEFINITION. An {\textit{extension}} of a conformal algebra $A$ is an ordered pair $(B, \sigma)$ where $B$ is a conformal algebra and 
$\sigma$ is a homomorphism $B$ to $A$.

Any extension can be identified with an exact sequence
\[
 0 \rightarrow \Ker\sigma \rightarrow B \rightarrow A \rightarrow 0.
\]

The conformal algebra $A$ is called {\textit{splitting off}} in the extension $(B, \sigma)$ if $B = A'\otimes \Ker\sigma$ where $A'\subseteq B$ is a subalgebra isomorfic to $A$.

The extension $B$ called singular if $(\Ker\sigma \circ_n \Ker\sigma) = 0$ for every $n\in \mathbb Z_+$.

Let $M$ be a bimodule over the conformal algebra $A$ and $\varphi \in C^2(A, M)$. Construct a singular extension of the algebra $A$ by the bimodule $M$. Let $B$ be equal to the $H$-modules direct sum $A\oplus M$. The operation $*$ can be defined by the rule
\[
(a_1 + u_1) * (a_2 + u_2) = a_1 * a_2 + a_1 * u_2 + u_1 * a_2 + \varphi(a_1, a_2),
\]
where  $(a_1 + u_1), (a_2 + u_2)\in B$. Then $B$ is a conformal algebra. Denote constructed extension by $(A; M, \varphi)$.

Next results are proved in [1]:

\textbf{LEMMA.~~}\textit{ If $\varphi \in Z^2(A, M)$, then $B = (A; M, \varphi)$ is an associative conformal algebra.}

\textbf{THEOREM.~~}\textit{ A conformal algebra $A$ is splitted off in a singular extension $(B, \sigma)$ if and only if a cocycle
 $\varphi$ is trivial in $H^2(A, \Ker\sigma)$.}

\subsection{Formulation of the problem}

Earlier we determined the conformal algebra $Cend_n,~n\ge1$. Here we consider case $n=1$. The conformal subalgebra $Cend_{1,x} = xCend_1 \subset Cend_1$ is a free $H$-module with a basis given by elements 
$x^k = x\circ_0 x\circ_0 \dots \circ_0 x$ ($k$ times), $k \ge 1$.

A question has been investigated in the paper sounds in this way: do the conformal algebra $Cend_{1,x}$ split off in every singular extension?
In other words, whether we can construct for every exact sequence 
\[
 0 \rightarrow M \rightarrow C \rightarrow Cend_{1,x} \rightarrow 0,
\]
(where $M^2=0$ in sense of $\circ$-multiplication) such cross-section $Cend_{1,x}\to C$ which is a conformal algebras homomorphism.

\section{Main result}

As it is shown in 2.3 the splitting off problem for the conformal algebra $Cend_{1,x}$ in a singular extension with the nilpotent kernel
$M$ is in a tough connection with algebra's second Hochschild cohomology group with coefficients in $M$. 

The main result of the paper is next
  
\begin{theorem}\label{thm:Theorem1}
Let $Cend_{1,x}$ be the associative conformal algebra described back in the text, $M$ is a bimodule over it. Then $H^2(Cend_{1,x}, M) = 0$.
\end{theorem}

\begin{proof}
We prove the theorem in several steps.

\begin{lemma}\label{lem:Lemma1}
Arbitrary 2-cocycle $\varphi\in Z^2(Cend_{1,x}, M)$ is completely determined by elements $\varphi_t(x, x)$ and $\varphi_0(x, x^l)$, $t, l\ge 1$ laying down in the bimodule $M$. 
\end{lemma}

\begin{proof}
Use the 2-cocycle definition:
\begin{equation}\label{eq: cocycle}
x^k\circ_n\varphi_m(x^l,x^q) + \varphi_n(x^k, x^l\circ_m x^q) =
\sum\limits_{s\ge0}^{} \binom{n}{s}(\varphi_{m+s}(x^k\circ_{n-s}x^l, x^q) + \varphi_{n-s}(x^k, x^l)\circ_{m+s}x^q),
\end{equation}
where $n, m\in \mathbb Z_+$, $k,l,q \ge1 $.

For $n=0$, $k=1$:

\begin{equation}
\varphi_m(x^{l+1},x^q) = x\circ_0\varphi_m(x^l, x^q) + m!\binom{q}{m}\varphi_0(x, 
x^{l+q-m}) - \varphi_0(x, x^l)\circ_m x^q.
\end{equation}

Use induction and express every values of the cocycle through $\varphi_m(x,x^q)$, $m\ge0$, $q\ge 1$.

Consider \eqref{eq: cocycle} for $m=0$, $k=l=1$:
\begin{equation}
x\circ_n\varphi_0(x,x^q) + \varphi_n(x,x^{q+1}) = \varphi_n(x^2, x^q) + 
n\varphi_{n-1}(x, x^q) - \sum\limits_{s\ge0}^{}\binom{n}{s}\varphi_{n-s}(x, x)\circ_s x^q.
\end{equation}

From (3) one can see that $\varphi_n(x^2,x^q) = x\circ_0\varphi_n(x, x^q) + 
n!\binom{q}{n}\varphi_0(x, x^{q-n+1}) - \varphi_0(x, x)\circ_n x^q$. We substitute the result into expression (4):

\begin{center}
$\varphi_n(x,x^{q+1}) = x\circ_0\varphi_n(x, x^q) + n!\binom{q}{n}\varphi_0(x, x^{q-n+1}) 
- \varphi_0(x, x)\circ_n x^q + n\varphi_{n-1}(x, x^q) - 
\sum\limits_{s\ge0}^{} \binom{n}{s}\varphi_{n-s}(x, x)\circ_s x^q - x\circ_n\varphi_0(x,x^q).$
\end{center} 

Assuming $n=1$ and using induction on $q$ we obtain that $\varphi_1(x,x^{q+1})$ is expressed through $\varphi_t(x, x)$ and $\varphi_0(x, x^l)$ for  $q,l\ge1,~q\ge l$. Applying sequently induction on $n$ and on $q$ we obtain the same expression for each $\varphi_n(x,x^{q+1})$, $n,q\ge1$. As required.
\end{proof}

\begin{lemma}\label{lem:Lemma2}
Let $\varphi_1(x, x) = 0$ and $\varphi_0(x, x^l) = 0$ for every $l\ge1$. 
Then $\varphi = 0$.
\end{lemma}

\begin{proof}

Due to Lemma 1 it is enough to prove that $\varphi_1(x, x) = 0 $ implies $\varphi_l(x, x) = 0$.

Hereinafter designate $\varphi_m(x, x) = \varphi_m$, $(x \circ_m \cdot) = L_m\in \End (M)$, 
$(\cdot \circ_m x) = R_m\in \End (M)$.

Applying associativity identities we've rewritten useful formulas in the new designation: 

\begin{equation}\label{eq: perest1}
L_nL_m(u) = \sum\limits_{s=0}^{n} (_s^n)((x\circ_{n-s}x)\circ_{m+s} u) = 
L_0L_{m+n}(u) + nL_{m+n-1}(u),
\end{equation}

\begin{equation}\label{eq: perest2}
R_nR_m(u) = \sum\limits_{s=0}^{m} (-1)^s (_s^m)(u \circ_{m-s} (x\circ_{n+s} x)) = 
\begin{cases}
R_m(u), & n=1,\\
0, & n>1,
\end{cases}
\end{equation}

\begin{equation}\label{eq: perest3}
R_nL_m(u) = \sum\limits_{s=0}^{m} (-1)^s (_s^m)(x\circ_{m-s} (u\circ_{n+s} x)) 
= \sum\limits_{s=0}^{m} (-1)^s (_s^m)L_{m-s}R_{n+s}(u),
\end{equation}
for $u\in M$.

Using 2-cocycle identity \eqref{eq: cocycle} obtain following relation:
\begin{center}
$(x\circ_1\varphi_m(x, x) - x\circ_0\varphi_{m+1}(x, x)) - \varphi_m(x \circ_1 
x, x) + (\varphi_1(x, x\circ_mx) - \varphi_0(x, x\circ_{m+1}x)) - \varphi_1(x, x)\circ_mx = 0$
\end{center}

Since the $\circ$-multiplication act on $x^k\in Cend_{1, x}$ like a differentiation we obtain equalities:
\begin{equation}\label{eq: sootn1}
x\circ_0\varphi_2 = x\circ_1\varphi_1 - \varphi_1\circ_1x,
\end{equation}
\begin{equation}\label{eq: sootn2}
x\circ_0\varphi_{m+1} = x\circ_1\varphi_m - \varphi_m - \varphi_1\circ_mx,~m>1.
\end{equation}

An assumption $\varphi_1 = 0$ implies there are neither right part in \eqref{eq: sootn1} nor last summand in \eqref{eq: sootn2}.

For $n=1$, $m=0$ formula \eqref{eq: perest1} gives 
$L_1L_0 = ((x\circ_1x)\circ_0\cdot) + ((x\circ_0x)\circ_1\cdot) = L_0 + L_0L_1$. Applying it to \eqref{eq: sootn2} obtain next equality
\begin{center}
$L_0^{m-2}\varphi_m = (L_1 - (m-2))~\dots ~(L_1 - 1)\varphi_2$.
\end{center}

The locality axiom implies there is such $N$ that $\varphi_m = 0$ for each $m>N$.

Let $N>1$. Then for $m = N+1$:~$(L_1 - (N-1))~\dots ~(L_1 - 1)\varphi_2 = 0$.

Denoting $(L_1 - (N-2)) \dots (L_1 - 1)\varphi_2 =  \prod\limits_{l=1}^{N-2}(L_1 - l)\varphi_2= \chi_1 \in M$ we get a new
equalities system:
\begin{equation}
\begin{cases}
L_1\chi_1 =x\circ_1\chi_1 = (N-1)\chi_1,
\\
L_0\chi_1 = x\circ_0\chi_1 = 0.
\end{cases}
\end{equation}
Second equality is obtained from \eqref{eq: sootn1} by the sequently left multiplication on $(L_1-2),\dots ,(L_1-(N-1))$ and with the condition $(L_1-1)L_0 = L_0L_1$.

Show the system implies $\chi_1 = 0$. Indeed, using \eqref{eq: perest1}, the associativity identity and system (10):

$0 = (x\circ_2x)\circ_0\chi_1 = L_2L_0\chi_1 - 2L_1L_1\chi_1 + L_0L_2\chi_1 = 
L_0L_2\chi_1 - 2L_1L_1\chi_1 \Rightarrow  L_0L_2\chi_1 = 2L_1L_1\chi_1$
$(N-1)^2\chi_1 = L_1L_1\chi_1 = (x\circ_1x)\circ_1\chi_1 + 
(x\circ_0x)\circ_2\chi_1 = L_1\chi_1 + L_0L_2\chi_1 
= (N-1)\chi_1 +2(N-1)^2\chi_1 \Rightarrow 0 = ((N-1) + (N-1)^2)\chi_1 = N(N-1)\chi_1$.

As long as $N>1$ last element in the chain of equalities is zero if only $\chi_1 = 0$.

Let $\chi_{k-1} = \prod\limits_{l=1}^{N-k}(L_1 - l)\varphi_2 = 
(L_1 - (N-k))\chi_k$, $k=2,\dots,N~(\chi_{N-1} = \varphi_2)$. 
Suppose reckoning above as the base of induction. Let $\chi_{k-1} = 0$ for certain $k$, $1<k<N$. Then
\begin{equation}
\begin{cases}
L_1\chi_k = (N-k)\chi_k,
\\
L_0\chi_1 = 0.
\end{cases}
\end{equation}

By similar actions one can obtain
\[
0 = (N-(k-1))(N-k)\chi_k = 
(N-(k-1))(N-k)\prod\limits_{l=1}^{N-(k+1)}(L_1 - l)\varphi_2.
\]
 hence $\chi_k = 0$.

Therefore $\chi_k = 0$ for $k=1,\dots,N-1$ and particularly $\chi_{N-1} = \varphi_2 = 0$. 

Now, with the condition $\varphi_2 = 0$ we have \eqref{eq: sootn2} rewritten in next way
\[
L_0\varphi_3 = 0, \quad L_0\varphi_{m+1} = (L_1-1)\varphi_m
\]
and with the same reflections we get $\varphi_3 = 0$ and so on until $\varphi_{N+1} = 0$ by the locality axiom. This is it, $\varphi_m = 0$ for every $m\in \mathbb Z_+$.
\end{proof}

\begin{lemma}\label{lem:Lemma3}
Consider a cocycle $\varphi \in Z^2(Cend_{1,x}, M)$. Assume exists such $\psi_1\in M$ that
\begin{equation}\label{eq: diff1}
 \varphi_1(x, x) = x\circ_1\psi_1 - \psi_1 + \psi_1\circ_1x.
\end{equation}
Then $\varphi \in B^2(Cend_{1,x}, M)$.
\end{lemma}

\begin{proof}
Define by induction $\psi_l$, $l\ge2$
\begin{equation}\label{eq: diff2}
\psi_{l+1} = - \varphi_0(x, x^l) +  x\circ_0\psi_l + \psi_1\circ_0x^l, \quad l\ge 1.
\end{equation}

Construct a 1-cochain with values $\psi(x^l) = \psi_l$. Consider $\delta\psi = \varphi' \in B^2(Cend_{1,x}, M)$. 
Then \eqref{eq: diff1} implies $\varphi_1'(x, x) = \varphi_1(x, x)$ and from \eqref{eq: diff2} follows  
$\varphi_0'(x, x^l) = \varphi_0(x, x^l)$.

Denote $\varphi - \varphi' = \pi \in Z^2(Cend_{1,x}, M)$. Then $\pi_1(x, x) = 0$ by the hipothesis, $\pi_0(x, x^l) = 0$ by the choice of $\psi_l$ and then regarding Lemma \ref{lem:Lemma2} we have $\pi = 0 $, i.e. $\varphi = \delta \psi$ as required.
\end{proof}

It is remained to show existance of the element $\psi_1$.

Firstly we derive a formula for $L^m_0\varphi_{m+1},~m\ge1$.

\begin{lemma}\label{lem:Lemma4}
If $\varphi \in Z^2(Cend_{1,x}, M)$ then
\begin{equation}\label{eq: dlinshtuka}
L_0^m\varphi_{m+1} = (L_1 - m)~\dots ~(L_1 - 2)L_1\varphi_1 - 
\sum\limits_{s=1}^{m}(L_1 - m)~\dots ~(L_1 - (s+1))L_0^{s-1}R_s\varphi_1.
\end{equation}
\end{lemma}

\begin{proof}
By induction on $m$. Using \eqref{eq: sootn1} for $m=1$ we obtain $L_0\varphi_2 = (L_1 - R_1)\varphi_1$. 
Then it is multiplied on the left by $(L_1 - 2)$. Applying \eqref{eq: perest1}: 
\[
 (L_1-2)(L_1-R_1)\varphi_1 = (L_1 - 2)L_0\varphi_2 = L_0(L_1-1)\varphi_2. 
\]
Substitute the result in \eqref{eq: sootn2}:
\[
L_0^2\varphi_3 = (L_1-2)(L_1-R_1)\varphi_1 - L_0R_2\varphi_1.
\]

$m \rightarrow m+1$: Let the formula be true for $L_0^{m-1}\varphi_m$. Then
\begin{multline}\nonumber
L_0^m\varphi_{m+1} = L_0^{m-1}(L_0\varphi_{m+1}) = L_0^{m-1}((L_1 - 1)\varphi_m 
- R_m\varphi_1) \\= (L_1 - m)L_0^{m-1}\varphi_m - L_0^{m-1}R_m\varphi_1 = 
(L_1 - m)((L_1 - (m-1)) \dots (L_1 - 2)L_1\varphi_1 \\ - \sum\limits_{s=1}^{m-1}(L_1 - (m-1))\dots (L_1 - 
(s+1))L_0^{s-1}R_s\varphi_1) - L_0^{m-1}R_m\varphi_1 \\ =  (L_1 - m) \dots (L_1 - 2)L_1\varphi_1 - 
\sum\limits_{s=1}^{m}(L_1 - m)\dots (L_1 - (s+1))L_0^{s-1}R_s\varphi_1 ,
\end{multline}
As required.
\end{proof}

Consider in this step multiplications $L_0^{k-1}L_k$. By induction on $k$ we will show that the equality 
\begin{equation}\label{eq: charf}
L_0^{k-1}L_k = L_1(L_1 - 1)~\dots ~(L_1 - (k-1))
\end{equation}
is fulfilled.

Using \eqref{eq: perest1} obtain

$k=2:~L_0L_2 = L_1^2 - L_1 = L_1(L_1 - 1)$;

$k\rightarrow k+1:~(L_1 - k)L_0^{k-1}L_k = L_0^{k-1}(L_1 - 1)L_k = -L_0^{k-1}L_k 
+ L_0^{k-1}(L_k + L_0L_{k+1}) = L_0^kL_{k+1}$.

By the hypothesis the formula is correct for $L_0^{k-1}L_k$. Obviously, $(L_1 - k)$ commutates with every elements of the same form.
So \eqref{eq: charf} is correct.

Now it's easy to prove next result.

\begin{lemma}\label{eq:Lemma5}
Let $V\subset M$ be a finite-dimensional space invariant with respect to $L_1$.Then there exists a decomposition in the direct sum of 
$L_1$-invariant subspaces $V = V_0\oplus V_1\oplus \dots \oplus V_k$ where $L_1 v = iv$ for $v\in V_i$.
\end{lemma}

\begin{proof}
Since $\dim V<\infty$ there exists only finite number of the basis elements $v_1,\dots ,v_n$.
Formula \eqref{eq: charf} for each $v_i $ gives:
\[
L_0^{k}L_{k+1}v_i = L_1(L_1 - 1) \dots (L_1 - k)v_i.
\]
The left part of the equality is zero for sufficiently large $k$ due to locality. Choose $k=\max(k_1,\dots ,k_n)$. Then 
$P_k = t(t - 1)\dots (t - k)$ is such polynomial that $P_k(L_1)V = 0$, i.e. $P_k(L_1)$ is the annihilator of the space $V$. This polynomial has no multiple roots thus the minimal polinomial of the operator $L_1$ on $V$ has no multiple roots likewise, i.e. $L_1$ is semisimple and every eigenvalues are laying down into $\mathbb Z_+$.

Then required decomposition $V = V_0\oplus V_1\oplus~\dots ~\oplus V_k$ is exists but generally speaking there might be zero spaces in the direct sum. 
\end{proof}

Consider the algebra of endomorphisms $\End M$ and subalgebra generated by operators $L_1$ and $L_0^mR_{m+1}$, $m\ge 0$:
\[
 A = \text{sub}\, \langle L_1, L_0^mR_{m+1} \mid m\ge 0 \rangle.
\] 

\begin{lemma}\label{lem:Lemma6}
Let $W = Au = \{au \mid a\in A\}$ be an $A$-module for arbitrary $u\in M$. Then $\dim W<\infty$.
\end{lemma}

\begin{proof}
It's easy to see that $W\subseteq W'$=span$\langle L_0^sL_1^mR_{n+s},~m,s\ge0,
~n\ge1\rangle$.
Indeed, it follows from \eqref{eq: perest1}, \eqref{eq: perest2}, \eqref{eq: perest3} 
that $W'$ is closed with respect to the action of generators of $A$. Let us show 
that $\dim W'<\infty$.

Firstly notice that the locality axiom implies $\dim V'<\infty$ where $V'$=span$\langle R_{n+s}(u)
,~s\ge0,~n>0\rangle$. After,
\[
L_1^m = L_1 + F(L_0L_2,\dots ,L_0^{m-2}L_{m-1}) + L_0^{m-1}L_m, 
\]
where $F$ is a linear function. Actually, using \eqref{eq: perest1} obtain 
$L_1^2 = L_1 + L_0L_2, L_1^3 = L_1(L_1^2) = L_1(L_1 + L_0L_2) = L_1^2 + L_1L_0L_2 = 
L_1 + L_0L_2 + L_0(L_1 + 1)L_2 = L_1 + 2L_0L_2 + L_0L_1L_2 = 
L_1 + 2L_0L_2 + L_0L_2 + L_0^2L_3 = L_1 + 3L_0L_2 + L_0^2L_3$ and so on. Hence elements
$L_1V'$ generate the same linear subspace as $L_0L_{k+1}V'$ which are finite dimensional
by virtue of the locality axiom. 
\end{proof}

For end of the proof of the theorem consider the algebra $A\subseteq \End(M)$ defined above and the $A$-module 
$V = A\varphi_1$. 

Let $\varphi_1\in M$ be a value of a 2-cocycle $\varphi$ and it fulfill the formula \eqref{eq: dlinshtuka}. For sufficiently large $m$ it turns to
\begin{equation}\label{eq: zanulili}
(L_1 - m)~\dots ~(L_1 - 2)L_1\varphi_1 = \sum\limits_{s=1}^{m}(L_1 - m)~\dots ~(L_1 - 
(s+1))L_0^{s-1}R_s\varphi_1
\end{equation}

As it shown back in the proof the operator $L_1$ divides the algebra $V$ on the finite number of invariant subspaces hence 
\[
\varphi_1 = v_0 + \dots  + v_m,
\]
where $L_1v_i = iv_i$.
Substitute that decomposition in \eqref{eq: zanulili}. Since there is a collection of nilpotent operators in the left part of the formula in result only $v_1$ remains:

\begin{center}
$\mp (m-1)! v_1 = \sum\limits_{s=1}^{m}(L_1 - m)~\dots ~(L_1 - (s+1))L_0^{s-2}R_{s-1}\varphi_1$
\end{center}

Use the operator $R_k$, $k\ge2$ to both part of the equality and obtain 
$R_k v_1 = Const\cdot R_k(\sum\limits_{s=2}^{m}(L_1 - m) \dots (L_1 - s)L_0^{s-2}R_{s-1} - L_0^{m-1}R_{m})\varphi_1$.

It follows from \eqref{eq: perest2} and \eqref{eq: perest3} that
$R_kv_1 = 0$, for $k\ge2$.

Initially find such $z\in V$ that $(L_1 - 1)z = v_0 + v_2 + v_3 \dots  + v_m = \varphi_1 - v_1$. It exists because $(L_1 - 1)$ annihilate only eigenspace $V_1$ and is non-degenerate in other. Choose a 1-cochain $\xi$ such
that its value in a bimodule $M$ is $z$, i.e. $\xi(x) = z$. 
Denote $\varphi' = \varphi - \delta\xi \in Z^2(Cend_{1,x}, M)$. Since \eqref{eq: perest2}, \eqref{eq: perest3} the cocycle $\varphi'$ has a property:
\[
R_k\varphi'_1 = R_k(\varphi_1 - (L_1 + R_1 - 1)z) = R_k (\varphi_1 - (L_1 - 1)z 
- R_1z) = R_k (v_1 - R_1z) = 0
\]
for $k\ge2$.

Without loss of generality one can assume \begin{equation}R_k\varphi_1 = 
0,~k\ge2.\end{equation}

Applying \eqref{eq: perest1}, \eqref{eq: perest2}, \eqref{eq: perest3} to a decomposition of $a\varphi_1\in A\varphi_1$, $a\in A$ we gather all operators $R_k$ at the right. Obtain $V = A\varphi_1 \subseteq \mathbb F[L_1,L_0] R_1\varphi_1$. It's easy to check using \eqref{eq: perest3}, (17) that $R_k V = 0$ for each $ k \ge 2$.

Relying on this result, show there is $\psi_1$ satisfying the equality $(L_1 + R_1 - 1)\psi_1 = \varphi_1$. 

Notice the identity $[L_k, R_1] = 0$, $k\ge 0$ is fulfilled in $\End (V)$ ($[\cdot,\cdot]$ is a common commutator).
Indeed, for $k=0$ endomorphisms commutate. For other $k$ it follows because of $R_kV = 0$:
\begin{multline}\label{eq: tozhdestvo}
[L_k,R_1]v = x\circ_k(v\circ_1 x) - (x\circ_k v)\circ_1 x \\
= x\circ_k(v\circ_1 x) - \sum\limits_{l=0}^{k}(-1)^l \binom{k}{l}x\circ_{k-l}(v\circ_{1+l}x) \\
= x\circ_k(v\circ_1 x) - x\circ_k(v\circ_1 x) +\sum\limits_{l=1}^{k}(-1)^{l-1}\binom{k}{l}L_{k-l}R_{1+l}v = 0.
\end{multline}

Also \eqref{eq: perest2} implies that $R_1$ acts on $M$ as the projection operator, i.e. $R_1^2 = R_1$. Hence the space $V$ is decomposed in the direct sum of the kernel $V^{(0)}$ and fixed points $V^{(1)}$. It follows from \eqref{eq: tozhdestvo} that $V^{(0)}$ and $V^{(1)}$ are closed with respect to the action of $L_1$.

Having the decomposition $\varphi_1 = \varphi_1^{(0)} + \varphi_1^{(1)}$ it's enough to show there exists such 
$\psi^{(0)},~\psi^{(1)}$ that 
\begin{equation}
\begin{gathered}
(L_1 + R_1 - 1)\psi^{(0)} = (L_1 - 1)\psi^{(0)} = \varphi^{(0)}_1,
\\
(L_1 + R_1 - 1)\psi^{(1)} = L_1\psi^{(1)} = \varphi^{(1)}_1.
\end{gathered}
\end{equation}

Values of the cocycle $\varphi_1$ satisfies to \eqref{eq: dlinshtuka} and (19) hence

\begin{multline}\nonumber
0 = (L_1 - m)\dots (L_1 - 2)(L_1 - R_1)\varphi_1 \\
 =  (L_1 - m) \dots (L_1 - 2)(L_1 - R_1)\varphi_1^{(0)} +  (L_1 - m)\dots (L_1 - 2)(L_1 - R_1)\varphi_1^{(1)} \\
 = (L_1 - m) \dots (L_1 - 2)L_1\varphi_1^{(0)} + (L_1 - m)\dots (L_1 - 2)(L_1 - 1)\varphi_1^{(1)}.
\end{multline} 

First summand is the element of $V^{(0)}$, second summand is the element of $V^{(1)}$ and the sum might be zero if and only if each summand is zero:
\[
 \begin{gathered}
 (L_1 - m) \dots (L_1 - 2)L_1\varphi_1^{(0)} =0,\\
 (L_1 - m)\dots (L_1 - 2)(L_1 - 1)\varphi_1^{(1)} =0.
 \end{gathered}
\]
Therefore $\varphi_1^{(0)} \in V_0\oplus V_2\oplus \dots \oplus V_k$ and the operator $L_1-1$ is non-degenerate in this direct sum of spaces. It implies there exists such $\psi^{(0)}\in V$ that $(L_1 - 1)\psi^{(0)} = \varphi^{(0)}_1$.

By the same way, $\varphi_1^{(1)}\in V_1\oplus \dots \oplus V_k$ and $L_1$  is non-degenerate there and hence exists $\psi^{(1)}\in V$ such that $L_1\psi^{(1)} = \varphi^{(1)}_1$. 

Hereby it is shown existance of the element $\psi_1=\psi^{(0)}+\psi^{(1)}$ which is required for the Lemma \ref{lem:Lemma3}. The theorem is proved.
\end{proof}

\begin{center}
Bibliography:
\end{center}

\begin{enumerate}
\item\textit{{Dolguntseva I.~A.}} Hochschild cohomologies for associative conformal algebras // Algebra and logic. 2007. 
V. 46, N.6. P. 688–706.

\item\textit{{Dolguntseva I.~A.}} Triviality of the second cohomology group of the conformal algebras  $Cend_n$ and  $Cur_n$ //  St. Petersburg Math. J. 2010. V.21, P. 53-63. 

\item\textit{Bakalov B., D'Andrea A., Kac V. G.} Theory of finite pseudoalgebras 
// Adv. Math. 2001. V.162, N.1. P. 1-140.

\item\textit{Bakalov B., Kac V. G., Voronov A.} Cohomology of conformal algebras 
// Comm. Math. Phys. 1999. V. 200, N.3. P.561-598.

\item\textit{Belavin A. A., Polyakov A. M., Zamolodchikov A. B.} Infinite 
conformal symmetry in two-dimansional quantum field theory // Nucl. Phis. B. 
1984. V. 241. P. 333-380.

\item\textit{Borcherds R. E.} Vertex algebras, Kac - Moody algebras, and the 
Monster // Proc. Nat. Acad. Sci. U.S.A. 1986. V. 83. P. 3068-3071.

\item\textit{Cheng S.-J., Kac V. G.} Conformal modules // Asian J. Math. 1997. 
V. 1. P. 181-193.

\item\textit{D'Andrea A., Kac V. G.} Structure theory of finite conformal 
alhebras // Selecta Math. New. Ser.1998. V. 4. P. 377-418.

\item\textit{Kac V. G.} Vertex algebras for beginners. Second edition. 
Providence, RI: AMS,1998. (University Lecture Series, vol. 10).

\item\textit{Kolesnikov P. S.} Associative conformal algebras with finite 
faithful representation // Adv. Math. 2006. V. 202, N.1. P. 602-637.

\item\textit{Kolesnikov P. S.} On the Wedderburn principal theorem in conformal 
algebras // Journal of Algebra and Its Applications.
 2007. V. 6, N.1. P. 119-134.
 
\item\textit{Zelmanov E. I.} On the structure of conformal algebras // Proc. / 
Intern. Conf. on Combinatorial and Computational Algebra. Hong Kong, China. May 
24-29, 1999. Providence, RI: AMS, 2000. P. 139-153. (Cont. Math., vol. 264).

\item\textit{Zelmanov E. I.} Idempotentns in conformal algebras // Proc. / Third 
Intern. Alg. Conf. Tainan, Taiwan. June 16-July 1, 2002. Ed. by Y. Fong et al. 
Dordrecht: Kluwen Acad. Publ., 2003. P. 257-266.
\end{enumerate}

\end{document}